\documentclass[12pt]{article}
\usepackage{amsmath,amssymb,graphicx,stmaryrd}
\usepackage{amsthm,mathrsfs}
\usepackage{ascmac}
\usepackage[all]{xy}
\usepackage{comment}

  \title{The Kac-Moody central extension on a loop group in the seven-term exact sequence}
  \author{TAKUYA FUJITANI}
  \date{\today}

\makeatletter

  \@addtoreset{equation}{section}

  \theoremstyle{plain}
  \newtheorem{axm}{公理}[section]
  \newtheorem{thm}[axm]{Theorem}
  \newtheorem{prp}[axm]{Proposition}

  \newtheorem*{clm*}{主張}

  \theoremstyle{definition}
  \newtheorem{dfn}[axm]{Definition}

  \theoremstyle{remark}
  \newtheorem*{note*}{記号}
  \newtheorem*{rmk*}{注釈}
  \newtheorem*{exm*}{例}
  \newtheorem{exm}[axm]{Example}
  \newtheorem{rmk}[axm]{Remark}

  \def\Z{\mathbb{Z}}
  
  \def\R{\mathbb{R}}

  \def\sgn{\mathrm{sgn}}

  \def\Hom{\mathrm{Hom}}

  \newcommand{\frk}[1]{\mbox{$\mathfrak{#1}$}}
  \newcommand{\Rm}[1]{\mbox{$\mathrm{#1}$}}

\begin{document}

\maketitle

\begin{abstract}
Mickelsson defined a group 2-cocycle on $C^{\infty}(D; G)$ and constructed a smooth central extension on a loop group $LG$, called the affine Kac-Moody central extension. 
We reformulate this central extension theory in the terms of the seven-term exact sequence for cohomology groups for discrete groups. 
\end{abstract}

\tableofcontents


\section{Introduction}

Let $G$ be a compact, connected, simply-connected, simple Lie group. 
Denote by $LG$ a (free) loop group $C^{\infty}(S^{1}; G)$ with pointwise product. 
There exists a central $U(1)$-extension $P$ on the loop group $LG$, called the universal central extension on $LG$. 
The universality means that, for any central extension $0 \to A \to E \to LG \to 1$, there exists a homomorphism $f : P \to E$ satisfying the following commutative diagram: 
\[\xymatrix{
0 \ar[r] & U(1) \ar[r] \ar[d]^{f|_{U(1)}} & P \ar[r] \ar[d]^{f} & LG \ar[r] \ar@{}[d]|{=} & 1\\
0 \ar[r] & A \ar[r] & E \ar[r] & LG \ar[r] & 1. 
}\]

Mickelsson \cite{Mickelsson1987} constructed the universal central extension on $LG$. 
Denote by $G^{D}$ the group of smooth maps $C^{\infty}(D; G)$ form the closed unit 2-disk $D$ to the Lie group $G$ with pointwise product. 
Set $G^{D}_{S^{1}} = \{ g \in G^{D} \mid g|_{\partial D = S^{1}} \equiv 1 \in G\}$ a normal subgroup of $G^{D}$. 
Then we have a group short exact sequence $1 \to G^{D}_{S^{1}} \to G^{D} \to LG \to 1$. 
He introduced a group 2-cocycle $C : G^{D} \times G^{D} \to \R$ on $G^{D}$ and a group 1-cochain $\Lambda : G^{D}_{S^{1}} \to \R/\Z$ which satisfy that $C|_{G^{D}_{S^{1}}} = -\delta\Lambda$; 
\[
C(f, g) = \frac{1}{8\pi^{2}}\int_{D}\Rm{tr}(f^{-1}df \wedge dg g^{-1}), \quad \Lambda(h) = -\frac{1}{24\pi^{2}}\int_{B}\Rm{tr}(\bar{h}^{-1}d\bar{h})^{\wedge 3}, 
\]
and constructed the central $U(1)$-extension $P$ on $LG$ by a quotient space of the product space $G^{D} \times U(1)$ by the equivalence relation; 
\[
(g, \lambda) \sim (g h, \lambda + C(g, h) + \Lambda(h))
\]
for $g \in G^{D}$, $h \in G^{D}_{S^{1}}$ and $\lambda \in \R/\Z \cong U(1)$. 
A group structure on $P$ is given by $(g, \lambda)(h, \mu) = (g h, \lambda + \mu + C(g, h))$ for $g, h \in G^{D}$ and $\lambda, \mu \in U(1)$. 
This product on $P$ is well-defined since the following equation holds: 
\begin{equation}
\label{eq:0}
C(g^{-1}h, g) - C(g^{-1}, h) + \Lambda(g^{-1} h g) - \Lambda(h) \equiv 0 \bmod \Z. 
\end{equation}

The infinitesimal version of group 2-cocycle $C : G^{D} \times G^{D} \to \R$ defines a Lie algebra 2-cocycle $-c^{D} : \frk{g}^{D} \times \frk{g}^{D} \to \R$ on a current algebra $\frk{g}^{D} = C^{\infty}(D; \frk{g})$: 
\[
c^{D}(X, Y) = \frac{1}{4\pi^{2}}\int_{D}\Rm{tr}(dX \wedge dY) = \frac{1}{4\pi^{2}}\int_{\partial D = S^{1}}\Rm{tr}(X dY)
\]
for $X, Y \in \frk{g}^{D}$. 
This Lie algebra 2-cocycle defines a Lie algebra 2-cocycle $c : L\frk{g} \times L\frk{g} \to \R$ on a loop algebra $L\frk{g} = C^{\infty}(S^{1}; \frk{g})$, and a Lie algebra central extension on $L\frk{g}$ associated with this 2-cocycle is well-known as the {\it affine Lie algebra} or the {\it Kac-Moody algebra} on the Lie algebra $\frk{g}$. 
Therefore the Lie group $P$ is also called the {\it Kac-Moody group} on the Lie group $G$. 

With respect to a group short exact sequence $1 \to G^{D}_{S^{1}} \to G^{D} \to LG \to 1$, we have an exact sequence of group cohomology called the seven-term exact sequence: 
\begin{align*}
0 \to H^{1}(LG; U(1)) \to H^{1}(G^{D}; U(1)) \to H^{1}(G^{D}_{S^{1}}; U(1))^{G^{D}}\\
 \to H^{2}(LG; U(1)) \to \Rm{Ker}(H^{2}(G^{D}; U(1)) \to H^{2}(G^{D}_{S^{1}}; U(1)))\\
 \to H^{1}(LG; H^{1}(G^{D}_{S^{1}};U(1))) \to H^{3}(LG; U(1)). 
\end{align*} 
We can rephrase the Mickelsson's 2-cocycle $C$ with the group 1-cochain $\Lambda$ and the universal central extension $P$ on the loop group $LG$ in terms of the above seven-term exact sequence. 

\begin{thm}
The cohomology class associated with the universal central extension $P$ on $LG$ is in the fourth term, and it is sent to one of the Mickelsson's 2-cocycle $C$ in the fifth term. 
Furthermore, equation \ref{eq:0} implies that the class of $C$ in the fifth term vanishes in the sixth term. 
\end{thm}

This paper is organized as follows. 
In section 2, we recall the definition and some properties of the cohomology theory for groups. 
It is known that the low-dimensional cohomology group for groups has the another definition. 
There is a bijection between the second cohomology group for groups and group extensions, in particularly, central extensions. 
In section 3, we introduce exact sequences with respect to a group short exact sequence. 
For the group short exact sequence, there is a spectral sequence of cohomology groups for groups, called the Lyndon-Hochschild-Serre (LHS) spectral sequence. 
We obtain an exact sequence of low-dimensional cohomology groups for groups, called the five-term exact sequence, and we can improve this sequence to the seven-term exact sequence. 
In section 4, we briefly review a construction of the universal central extension on a free loop group due to Mickelsson. 
In section 5, we show what the Mickelsson's construction means in terms of the seven-term exact sequence for a group short exact sequence. 

\section*{Acknowledgement}
I would like to express my deepest gratitude to Prof. Moriyoshi who provided carefully considered feedback and valuable comments.


\section{The group cohomology and the group extensions}

In this section, we recall the definition and some properties of cohomology theory for groups. 

Let $G$ be a group and $A$ be an Abelian group. 
Assume that $A$ has a right $G$-action that is written by $a.g \in A$ for $a \in A$ and $g \in G$. 

\begin{dfn}
The {\it group cochain complex} $(C^{\ast}(G; A), \delta)$ with coefficients in $A$ is given by the pair 
\[
C^{p}(G; A) = \{c : G^{p} \to A\}, \quad \delta : C^{p}(G; A) \to C^{p+1}(G; A), 
\]
where a {\it group} $p$-{\it cochain} $c : G^{p} \to A$ is an arbitrary function on the $p$-tuple product of $G$ and where the coboundary $\delta$ is defined to be 
\begin{align*}
\delta c(g_{1}, \ldots, g_{p+1}) = c(g_{2}, \ldots, g_{p+1}) & + \sum_{i=1}^{p}(-1)^{i}c(g_{1}, \ldots, g_{i} g_{i+1}, \ldots, g_{p+1})\\
& + (-1)^{p+1} c(g_{1}, \ldots, g_{p}).g_{p+1}. 
\end{align*}
Denote by $H^{\ast}(G; A)$ the cohomology group of $C^{\ast}(G; A)$, called the {\it group cohomology} of $G$ with coefficients in $A$. 
\end{dfn}

Denote by $C^{p}(G; A)_{N}$ the group of the {\it normalized} group $p$-cochains $c : G^{p} \to A$ which are group $p$-cochains such that $c(g_{1}, \ldots, g_{p}) = 0$ whenever one of the $g_{i}$ is equal to $1$. 
As can be seen easily, if $c$ is a normalized cochain then $\delta c$ is so. 
Then $(C^{p}(G; A)_{N}, \delta)$ is a subcomplex of $C^{\ast}(G; A)$.  
As is well known, the cohomology of the normalized group cochain complex $C^{\ast}(G; A)_{N}$ coincides the one of the ordinary group cochain complex $C^{\ast}(G; A)$. 

\begin{exm}
Assume that $G$ acts on $A$ trivially. 
Then every group 1-cocyle $f  \in C^{1}(G; A)$ satisfies that 
\[
\delta f(g, h) = f(h) - f(g h) + f(g) = 0
\]
for $g, h \in G$, so that, $f$ is a homomorphism. 
On the other hand, for any 0-cochain $c \in C^{0}(G; A)$, its coboundary is $\delta c(g) = c.g - c = 0$ for $g \in G$. 
This implies that all of the group 1-coboundaries are trivial. 
Then the first cohomology group $H^{1}(G; A)$ is just the group of homomorphisms $\Hom(G, A)$. 
\end{exm}

As with the first cohomology, there is an alternative definition for the second cohomology groups. 
We recall group extensions and central extensions. 

\begin{dfn}
Let $A$, $\Gamma$ and $G$ be arbitrary groups. 
A {\it group extension} $\Gamma$ of $G$ by $A$ is a short exact sequence of groups: 
\[
1 \to A \to \Gamma \to G \to 1. 
\]
Furthermore, $\Gamma$ is a {\it central extension} if $A \to \Gamma$ factors through the center of $\Gamma$. 
\end{dfn}

Consider the group extension $\Gamma$ of $G$ by an Abelian group $A$: 
\[
0 \to A \hookrightarrow\ \Gamma \to G \to 1
\]
Set $s : G \to \Gamma$ an arbitrary set-theorical section. 
There is a group action of $G$ on $A$ defined by 
\[
a.g = s(g)^{-1} a s(g)
\]
for $g \in G$ and $a \in A$ since $p(a.g) = g^{-1} p(a) g = 1$, that is, $a.g \in A$. 
In fact, this group action is independent of the choice of section $s$. 

For $g, h \in G$, we have $s(g) s(h) s(g h)^{-1} \in A$ because of $p(s(g) s(h) s(g h)^{-1}) = 1$. 
Then we can define a group 2-cochain $\chi \in C^{2}(G; A)$ as the following: 
\[
\chi(g, h) = s(g) s(h) s(g h)^{-1} \in A. 
\]
It is easy to check that $\chi \in C^{2}(G; A)$ is a group 2-cocycle and that its cohomology class $[\chi] \in H^{2}(G; A)$ is independent on the choice of section $s : G \to \Gamma$, so that, the class depends only on the group extension. 

\begin{dfn}
The {\it extension class} of the group extension $\Gamma$, denote by $e(\Gamma)$, is the cohomology class $[\chi] \in H^{2}(G; A)$. 
\end{dfn}

If a section $s : G \to \Gamma$ is a homomorphism then the corresponding 2-cocycle $\chi \in C^{2}(G; A)$ is obviously trivial. 
This implies that the splitting extension has the trivial extension class. 
The following is well known. 

\begin{prp}[see Brown \cite{Brown1982}]
Let $G$ be a group and $A$ be a $G$-module. 
The second cohomology group $H^{2}(G; A)$ is to the equivalence classes of central extensions of $G$ by $A$; 
\[
H^{2}(G; A) \cong \{\mbox{central extensions of $G$ by $A$}\}/\{\mbox{splitting extensions}\}. \qed
\]
\end{prp}

\begin{exm}[the universal covering of the circle diffeomorphisms.]
\label{exm:univ_cov}
Set $\Rm{Diff}^{+}(S^{1})$ the group of orientation-preserving diffeomorphisms on the circle $S^{1}$. 
Let $H$ be the universal covering group of $\Rm{Diff}^{+}(S^{1})$. 
Then we obtain a central extension $0 \to \Z \to H \xrightarrow{\pi} \Rm{Diff}^{+}(S^{1}) \to 1$. 
By the inclusion $\Z \hookrightarrow \R$, we extend this central extension to the central extension by $\R$: 
\[
0 \to \R \to H_{\R} \to \Rm{Diff}_{+}(S^{1}) \to 1. 
\]
As is well known, the extension class $e(H_{\R})$ is given by the formula 
\[
\chi(g_{1}, g_{2}) = \frac{1}{4\pi^{2}} \int_{0}^{2\pi}\left(h_{1} \circ h_{2}(x) - h_{1}(x) - h_{2}(x)\right) \, dx
\]
with $h_{1}, h_{2} \in H$ and $g_{1} = \pi(h_{1})$, $g_{2} = \pi(g_{2})$. 
\end{exm}

\begin{dfn}
Let $G$ be a group and $A$ be a $G$-module. 
For a $p$-cochain $c : G^{p} \to A$ and a $q$-cochain $c' : G^{q} \to A$, define a {\it cup product} $c \cup c' : G^{p+q} \to A \otimes A$ as 
\[
(c \cup c')(g_{1}, \ldots, g_{p+q}) = c(g_{1}, \ldots, g_{p}).g_{p+1} \cdots g_{p+q} \otimes c'(g_{p+1}, \ldots, g_{p+q}) \in A \otimes A. 
\]
\end{dfn}

It can be seen immediately that the cup product $c \cup c'$ is a cocycle if two cochains $c$ and $c'$ are cocycles, so that, we can define cup product on cohomology groups. 

\begin{prp}
\label{prp:1.8}
Let $G$ be a group, $V$ and $A$ be $G$-modules. 
Take a group 1-cocycle $\bar{\theta} : G \to V$ and a $G$-equivariant bilinear form $\omega : V \otimes V \to A$. 
Then a group 2-cochain $\omega(\bar{\theta} \cup \bar{\theta}) : G \times G \to A$ is a group 2-cocyle. 
\end{prp}

\begin{proof}
According to $G$-equivariance of $\omega$, the bilinear map $\omega$ is compatible with the group coboundary $\delta$. 
Hence we obtain 
\[
\delta \omega(\bar{\theta} \cup \bar{\theta}) = \omega(\delta(\bar{\theta} \cup \bar{\theta})) = \omega(\delta\bar{\theta} \cup \bar{\theta} - \bar{\theta} \cup \delta\bar{\theta}) = 0. 
\]
\end{proof}

\begin{exm}
Recall the group of orientation-preserving diffeomorphisms $\Rm{Diff}_{+}(S^{1})$ on the circle $S^{1}$. 
Define a group 1-cocycle $\ell : \Rm{Diff}_{+}(S^{1}) \to C^{\infty}(S^{1}; \R)$ as $\ell(g) = \log (g')$, and a $\Rm{Diff}_{+}(S^{1})$-invariant bilinear form $\omega : C^{\infty}(S^{1}; \R) \otimes C^{\infty}(S^{1}; \R) \to \R/\Z$ as 
\[
\omega(f, g) = \frac{1}{2}\int_{S^{1}} f \, dg. \bmod \Z
\]
Then a group 2-cocycle $\omega(\ell \cup \ell) \in C^{2}(\Rm{Diff}_{+}(S^{1}); \R/\Z)$ is well-known as the {\it Bott-Virasoro cocycle} (see Khesin-Wendt \cite{KW2009}). 
\end{exm}


\section{The LHS spectral sequence and the five/seven-term exact sequence}
In this section, we introduce the Lyndon-Hochschild-Serre spectral sequence and the five- (or seven-)term exact sequence. 
We consider group cochain complexes as normalized ones through this section. 

Let $G$ be a group, $N$ be a normal subgroup of $G$, and $A$ be a $G$-module. 
Then there is a group extension:
\[
1 \to N \to G \to G/N \to 1. 
\]

We define a filtration $(C_{p}^{\ast})_{p\in\Z}$ of the complex $C^{\ast}(G; A)$ as follows: 
\[
\begin{cases}
C_{p}^{q} = C^{q}(G; A), & \mbox{if $p \leq 0$}, \\
C_{p}^{q} = 0, & \mbox{if $q < p$}, 
\end{cases}
\]
and for $0 < p \leq q$, we set $C_{p}^{q}$ the group of $q$-cochains $c \in C^{q}(G; A)$ such that $f(g_{1}, \ldots, g_{p}) = 0$ whenever $(n-j+1)$ of the arguments belong to $N$. 
We can check $\delta(C_{p}^{\ast}) \subset C_{p}^{\ast}$ easily. 

This filtration defines a spectral sequence $E_{r}^{p, q}$, called the {\it Lyndon-Hochschild-Serre spectral sequence} for the group extension: 
\[
1 \to N \to G \to G/N \to 1. 
\] 

\begin{dfn}[Hochschild-Serre \cite{Hochschild-Serre53}]
Let $G$ be a group, $N$ be a normal subgroup of $G$ and $A$ be a $G$-module. 
We set $Z_{r}^{p, q} = \{ c \in C_{p}^{p+q} \mid \delta c \in C_{p+r}^{p+q+1}\}$. 
The Lyndon-Hochschild-Serre (LHS) spectral sequence is a spectral sequence defined as
\[
E_{r}^{p, q} = Z_{r}^{p, q}/(Z_{r-1}^{p-1, q+1} + \delta Z_{r-1}^{p+1-r, q+r-2})
\]
and differential operators $d_{r} : E_{r}^{p, q} \to E_{r}^{p+r, q-r+1}$ are induced by $\delta$. 
\end{dfn}

\begin{prp}[Hochschild-Serre \cite{Hochschild-Serre53}]
The LHS spectral sequence $E_{r}^{p, q}$ satisfies the follows. 
\begin{enumerate}
\item 
There exists an isomorphism 
\[
E_{1}^{p, q} \cong C^{p}(G/N; H^{q}(N; A)). 
\]
and the differential $d_{1}$ corresponds to the group coboundary map $\delta_{G/N} : C^{p}(G/N; H^{q}(N; A)) \to C^{p+1}(G/N; H^{q}(N; A))$. 
\item 
There exists an isomorphism
\[
E_{2}^{p, q} \cong H^{p}(G/N; H^{q}(N; A)). 
\]
\item 
This spectral sequence $E_{r}^{p, q}$ converges to $H^{p+q}(G; A)$, that is, there is an isomorphism
\[
E_{\infty}^{p, q} \cong \Rm{Im}(H^{p+q}(C_{p}) \to H^{p+q}(G; A))/\Rm{Im}(H^{p+q}(C_{p+1}) \to H^{p+q}(G; A)). 
\]
\end{enumerate}
\end{prp}

\begin{proof}[Sketch of proof]
Suppose $c \in C^{p+q}(G; A)$. 

Let $\sigma$ be a $(p, q)$-shuffle, that is, $\sigma$ is a permutation of $(1, \ldots, p+q)$ such that $\sigma(1) < \cdots \sigma(p)$ and $\sigma(p+1) < \cdots < \sigma(p+q)$. 
We define $c_{\sigma} \in C^{p+q}(G; A)$ by
\[
c_{\sigma}(\alpha_{1}, \ldots, \alpha_{q}, \beta_{1}, \ldots, \beta_{p}) = c(g_{1}, \ldots, g_{p+q}), 
\]
where
\[
g_{\sigma(i)} = 
\begin{cases}
\beta_{i} & \mbox{if $1 \leq i \leq p$}, \\
\beta_{\sigma(i) + p - i}^{-1} \cdots \beta_{1}^{-1} \alpha_{i-p} \beta_{1} \cdots \beta_{\sigma(i) + p - i}, & \mbox{if $p+1 \leq i \leq p+q$}
\end{cases}\]
and 
\[
c_{p} = \sum_{\sigma} \sgn(\sigma) c_{\sigma}. 
\]

Finally, we define a homomorphism $\varphi : C^{p+q}(G; A) \to C^{p}(G/N; C^{q}(N; A))$ such that 
\[
\varphi(c)(\beta_{1}, \ldots, \beta_{p})(\alpha_{1}, \ldots, \alpha_{q}) = c_{p}(\alpha_{1}, \ldots. \alpha_{q}, \beta_{1}, \ldots, \beta_{p})
\]
for $\beta_{1}, \ldots, \beta_{p} \in G/N$ and $\alpha_{1}, \ldots, \alpha_{q} \in N$. 
This homomorphism $\varphi$ induces isomorphisms $E_{1}^{p, q} \to C^{p}(G/N; H^{q}(N; A))$ and $E_{2}^{p, q} \to H^{p}(G/N; H^{q}(N; A))$. 
\end{proof}

Denote by $A^{N}$ the subgroup of $A$ consisting of all of $G$-invariant elements. 
It is easy to check that $H^{0}(N; A)$ is isomorphic to $A^{N}$. 
For a group $p$-cocycle $c \in C^{p}(G/N; A^{N})$, we define a new cocycle $\overline{c} \in C^{p}(G/N; A)$ by composing with natural projections $G \to G/N$ and a natural inclusion $A^{N} \hookrightarrow A$; 
\[
\overline{c} : G^{p} \to (G/N)^{p} \xrightarrow{c} A^{N} \hookrightarrow A. 
\]
This correspondence define a homomorphism $\Rm{inf} : H^{p}(G/N; A^{N}) \to H^{p}(G; A)$, which is called the {\it inflation homomorphism}. 

On the other hand, an inclusion map $N \hookrightarrow G$ induces a homomorphism $\Rm{res} : H^{q}(G; A) \to H^{q}(N; A)$. 
It is easy to check that the homomorphism $\Rm{res}$ factors through the $G$-invariant part of $H^{q}(N, A)$. 
This homomorphism $\Rm{res} : H^{q}(G; A) \to H^{0}(G/N; H^{q}(N; A))$ is called the {\it restriction homomorphism}. 

\begin{prp}[Hochschild-Serre \cite{Hochschild-Serre53}]
Let $1 \to N \to G \to G/N \to 1$ be a group extension and $A$ be a $G$-module. 
Then there exists a group exact sequence
\begin{align*}
0 \to H^{1}(G/A; A^{N}) \xrightarrow{\Rm{inf}} H^{1}(G; A) \xrightarrow{\Rm{res}} & H^{0}(G/N; H^{1}(N; A))\\
& \xrightarrow{d_{2}} H^{2}(G/N; A^{N}) \xrightarrow{\Rm{inf}} H^{2}(G; A). 
\end{align*}
\qed
\end{prp}

This exact sequence is called the {\it five-term exact sequence} or the {\it inflation-restriction exact sequence} for a group extension $1 \to N \to G \to G/N \to 1$. 
In general, given the spectral sequence converging to some cohomology group, it is well known that there is an exact sequence called the five-term exact sequence. 

Moreover, the five-term exact sequence can be extended to the {\it seven-term exact sequence}. 

\begin{prp}[Dekimpe-Hartl-Wauters \cite{DHW2012}]
Let $1 \to N \to G \to G/N \to 1$ be a group extension and $A$ be a $G$-module. 
Then there exists a group exact sequence
\[\xymatrix{
0 \ar[r] & H^{1}(G/N; A^{N}) \ar[r]^-{\Rm{inf}} & H^{1}(G; A) \ar[r]^-{\Rm{res}} & H^{1}(N; A)^{G}\\
 \ar[r]^-{d_{2}} & H^{2}(G/N; A^{N}) \ar[r]^-{\Rm{inf}} & \Rm{Ker}(H^{2}(G; A) \ar[r]^-{\Rm{res}} & H^{2}(N; A)^{G})\\
 \ar[r]^-{\rho} & H^{1}(G/N; H^{1}(N; A)) \ar[r]^-{d_{2}} & H^{3}(G/N; A^{N}) & 
}\]
\qed
\end{prp}

Suppose that a group 2-cocycle $C : G \times G \to A$ and a group 1-cocycle $\Lambda : N \to A$ satisfy $C|_{N} = -\delta\Lambda$. 
Then we define a map $\rho(C, \Lambda) : G/N \to C^{1}(N; A)$ as
\[
\rho(C, \Lambda)([g])(n) = \Lambda(g^{-1} n g) - \Lambda(n) + C(g^{-1}n, g) - C(g^{-1}, n). 
\]
This map defines the fifth homomorphism $\rho : \Rm{Ker}(\Rm{res}) \to H^{1}(G/N; H^{1}(N; A))$ (see Dekimpe-Hartl-Waures \cite{DHW2012}).


\section{Construction of the Kac-Moody central extension due to Mickelsson}

In this section, we review the construction of the universal central extension on a free loop group due to Mickelsson \cite{Mickelsson1987}. 

Put $G = \Rm{SU}(2)$. 
Denote by $LG$ a free loop group $C^{\infty}(S^{1}, G)$, and set $G^{D} = C^{\infty}(D, G)$ and $G^{D}_{S^{1}} = \{g \in G^{D} \mid g|_{S^{1} = \partial D} \equiv 1 \in G\}$ with pointwise products.  
Then we have a group central extension $1 \to G^{D}_{S^{1}} \to G^{D} \to LG \to 1$. 
Note that the group $G^{D}_{S^{1}}$ is isomorphic to a group $G^{S^{2}}_{\ast} = C^{\infty}_{\ast}(S^{2}; G)$ which consists of all base point-preserving smooth maps from the 2-sphere $S^{2}$ to $G$. 

Mickelsson \cite{Mickelsson1987} constructed a central $U(1)$-extension on $LG$ which is the {\it universal central extension} on $LG$. 
We shall start a group 2-cocycle on $G^{D}$. 

\begin{dfn}[the Mickelsson's 2-cocycle \cite{Mickelsson1987}]
A group 2-cochain $C : G^{D} \times G^{D} \to \R$ on $G^{D}$ is defined as
\[
C(g, h) := \frac{1}{8\pi^{2}}\int_{D}\Rm{tr}(g^{-1}dg \wedge dh \, h^{-1})
\]
with coefficients in $\R$. 
\end{dfn}

Recall that an additive group of the Lie algebra $\frk{g}$-valued 1-forms $\Omega^{1}(D; \frk{g})$ on the 2-disk $D$ is a right $G^{D}$-module by right adjoint action. 
We define a $G^{D}$-invariant bilinear form $\omega : \Omega^{1}(D; \frk{g}) \otimes \Omega^{1}(D; \frk{g}) \to \R$ as 
\[
\omega(\eta, \tau) = \frac{1}{8\pi^{2}}\int_{D}\Rm{tr}(\eta \wedge \tau), 
\]
and a group 1-cocycle $\bar{\theta} : G^{D} \to \Omega^{1}(D; \frk{g})$ with coefficients in $\Omega^{1}(D; \frk{g})$ as follow: 
\[
\bar{\theta}(g) = g^{\ast}\theta = g^{-1}dg \in \Omega^{1}(D; \frk{g}), 
\]
where $\theta \in \Omega^{1}(G; \frk{g})$ is the Maurer-Cartan 1-form on $G$. 
Then we obtain 
\[
\omega(\theta \cup \theta)(g, h) = \frac{1}{8\pi^{2}}\int_{D}\Rm{tr}(h^{-1}\bar{\theta}(g)h \wedge \bar{\theta}(h)) = C(g, h), 
\]
and it follows that the cochain $C$ is a 2-cocycle by proposition \ref{prp:1.8}. 

For a smooth map $\bar{g} : B \to G$ from the 3-ball $B = \{ x \in \R^{3} \mid \|x\| \leq 1\}$ to the group $G$, put 
\[
\Lambda(\bar{g}) = -\frac{1}{24\pi^{2}}\int_{B}\Rm{tr}(\bar{g}^{-1}d\bar{g})^{\wedge 3}. 
\]
Recall that a cohomology class of $\frac{1}{24\pi^{2}}\Rm{tr}(\theta^{\wedge 3}) \in \Omega^{3}(G)$ is a generator of $H^{3}(G; \Z)$. 
It follows that, if two smooth maps $\bar{g}, \bar{g}' : B \to G$ take the same value on the boundary $\partial B = S^{2}$, then $\Lambda(\bar{g}) - \Lambda(\bar{g}') \in \Z$. 
Hence $\Lambda(\bar{g}) \bmod \Z$ depends only on $g = \bar{g}|_{\partial B = S^{2}} : S^{2} \to G$, so that, we obtain a group 1-cochain $\Lambda : G_{S^{1}}^{D} \cong C^{\infty}_{\ast}(S^{2}; G) \to \R/\Z$

\begin{prp}
\label{prp:3.2}
For $g, h \in G^{D}_{S^{1}} < G^{D}$, the following equation holds: 
\[
-\delta\Lambda(g, h) = C(g, h) \bmod \Z. 
\]
\end{prp}

\begin{proof}
By the invariance of trace under the adjoint action and the cyclic permutations, we have 
\begin{align*}
& \Rm{tr}(g^{-1}dg)^{\wedge 3} - \Rm{tr}((gh)^{-1}d(gh))^{\wedge 3} + \Rm{tr}(h^{-1}dh)^{\wedge 3} \\
& = \Rm{tr}(g^{-1}dg)^{\wedge 3} - \Rm{tr}(h^{-1}(g^{-1}dg)h + h^{-1}dh)^{\wedge 3} + \Rm{tr}(h^{-1}dh)^{\wedge 3}\\
& = -3\Rm{tr}((g^{-1}dg)^{\wedge 2} \wedge (dh h^{-1})) +3\Rm{tr}((g^{-1}dg) \wedge (dh h^{-1})^{\wedge 2}). 
\end{align*}
Recall that Maurer-Cartan 1-form $\theta$ satisfies that $d\theta + \theta \wedge \theta = 0$. 
Therefore, we obtain 
\begin{align*}
& \Rm{tr}(g^{-1}dg)^{\wedge 3} - \Rm{tr}((gh)^{-1}d(gh))^{\wedge 3} + \Rm{tr}(h^{-1}dh)^{\wedge 3} \\
& = 3\Rm{tr}(d(g^{-1}dg) \wedge (dh h^{-1})) -3\Rm{tr}((g^{-1}dg) \wedge d(dh h^{-1}))\\
& = 3 d\Rm{tr}(g^{-1}dg \wedge dh h^{-1}). 
\end{align*}
This implies that the proposition is proven. 
\end{proof}

Furthermore, by directly calculation, we can obtain that, for $g \in G^{D}$ and $h \in G^{D}_{S^{1}}$,  
\begin{equation}
\label{eq:1}
C(g^{-1}h, g) - C(g^{-1}, h) + \Lambda(g^{-1} h g) - \Lambda(h) \equiv 0 \bmod \Z. 
\end{equation}

We will construct a principal $U(1)$-bundle $P$ on $LG$, and show that the space $P$ have a group structure (see Mickelsson \cite{Mickelsson1987}). 
We regard the unitary group $U(1)$ as the additive group $\R/\Z$.  
Consider the product space $G^{D} \times U(1)$ with the equivalence relation 
\[
(g, \lambda) \sim (g h, \lambda + C(g, h) + \Lambda(h))
\]
for $g \in G^{D}$, $h \in G^{D}_{S^{1}}$ and $\lambda \in U(1)$. 
Proposition \ref{prp:3.2} implies that the equivalence relation is well-defined. 
Then we define $P = (G^{D} \times U(1))/\sim$. 
The product of two equivalence classes on $P = (G^{D} \times U(1))/\sim$ is given by 
\[
[g, \lambda][h, \mu] = [g h, \lambda + \mu + C(g, h)]
\] 
for $g, h \in G^{D}$ and $\lambda, \mu \in U(1)$. 
Equation \ref{eq:1} implies that this product is well-defined on $P$. 

As above, we obtain a central $U(1)$-extension on the loop group $LG$: 
\[
0 \to U(1) \to P \to LG \to 1. 
\]

\begin{rmk}
This infinite-dimensional Lie group $P$ is known as the {\it Kac-Moody group} which corresponds the {\it Kac-Moody algebra} based on $\frk{g}$ (see Khesin-Wendt \cite{KW2009}). 
Thus this central extension is called the {\it Kac-Moody central extension} on $LG$. 
\end{rmk}

\begin{rmk}
The principal $U(1)$-bundle $P$ on $LG$ has the non-trivial topology, so that, the group 2-cocycle corresponding the Kac-Moody central extension on $LG$ is not smooth. 
\end{rmk}


\section{The Kac-Moody central extension in the seven-term exact sequence}

We consider the situation of the last section. 
We showed the Mickelsson's 2-cocycle on the group $G^{D}$ and the Kac-Moody central extension $P$ on the loop group $LG$ in the last section. 
We can rephrase them in terms of the seven-term exact sequence for the group short exact sequence $1 \to G^{D}_{S^{1}} \to G^{D} \to LG \to 1$. 

\begin{thm}
On the seven-term exact sequence for the group short exact sequence $1 \to G^{D}_{S^{1}} \to G^{D} \to LG \to 1$, i.e., 
\begin{align*}
0 \to H^{1}(LG; U(1)) \to H^{1}(G^{D}; U(1)) \to H^{1}(G^{D}_{S^{1}}; U(1))^{G^{D}}\\
 \to H^{2}(LG; U(1)) \to \Rm{Ker}(H^{2}(G^{D}; U(1)) \to H^{2}(G^{D}_{S^{1}}; U(1)))\\
 \to H^{1}(LG; H^{1}(G^{D}_{S^{1}};U(1))) \to H^{3}(LG; U(1)), 
\end{align*} 
the follows hold. 
\begin{enumerate}
\item 
The cohomology class of the Mickelsson's cocycle $C$ in $H^{2}(G^{D}; U(1))$ with a group 1-cochain $\Lambda : G^{D}_{S^{1}} \to G$ is in the fifth term. 

\item 
The image of the class $[C] \in H^{2}(G^{D}; U(1))$ to the sixth term $H^{1}(LG; H^{1}(G^{D}_{S^{1}}; U(1)))$ vanishes. 

\item 
The cohomology class of the Kac-Moody cantral extension on $LG$ in the fourth term $H^{2}(LG; U(1))$ is sent to the class $[C] \in H^{2}(G^{D}; U(1))$ in the fifth term. 
\end{enumerate}
\end{thm}

\begin{proof}
\begin{enumerate}
\item 
Proposition \ref{prp:3.2} implies that the class of the Mickelsson's 2-cocycle $C$ in $H^{2}(G^{D}; U(1))$ vanishes under the restriction homomorphism. 

\item 
Equation \ref{eq:1} implies that the class $[C]$ in the fifth term vanishes under the homomorphism $\rho$. 

\item 
By definition, the pullback of the Kac-Moody central extension $P$ by the restriction map $G^{D} \to LG$ is a central extension on $G^{D}$ corresponding the Mickelsson's 2-cocycle $C$. 
Hence the cohomology class corresponding the Kac-Moody central extension on $LG$ is sent to the class $[C]$ in the fifth term. 
\end{enumerate}
\end{proof}

\bibliography{reference}

\begin{thebibliography}{1}

\bibitem{Brown1982}
K.S. Brown.
\newblock {\em Cohomology of Groups}.
\newblock Graduate Texts in Mathematics. Springer, 1982.

\bibitem{DHW2012}
K.~Dekimpe, M.~Hartl, and S.~Wauters.
\newblock A seven-term exact sequence for the cohomology of a group extension.
\newblock {\em Journal of Algebra}, 369:70--95, nov 2012.

\bibitem{Hochschild-Serre53}
G.~Hochschild and J-P. Serre.
\newblock Cohomology of group extensions.
\newblock In {\em Transactions of the American Mathematical Society},
  volume~74, pages 110--134. Amer. Math. Soc., 1953.

\bibitem{KW2009}
B.~Khesin and R.~Wendt.
\newblock {\em The Geometry of Infinite-Dimensional Groups}, volume~51 of {\em
  Ergebnisse der Mathematik und ihrer Grenzgebiete. 3. Folge / A Series of
  Modern Surveys in Mathematics}.
\newblock Springer Berlin Heidelberg, 2009.

\bibitem{Mickelsson1987}
J.~Mickelsson.
\newblock Kac-moody groups, topology of the dirac determinant bundle, and
  fermionization.
\newblock {\em Communications in Mathematical Physics}, 110(2):173--183, jun
  1987.

\end{thebibliography}

\end{document}